\documentclass[11pt]{article}
\usepackage{t1enc}
\usepackage[latin1]{inputenc}
\usepackage[english]{babel}
\usepackage{amsmath,amsthm}
\usepackage{amsfonts}
\usepackage{latexsym}
\usepackage[dvips]{graphicx}
\usepackage{graphicx}
\usepackage{mathrsfs}
\usepackage{color}
\newcounter{countclaim}

\date{} \textwidth 16cm \textheight 22cm \topmargin 0 cm \hoffset
-1.5 cm \voffset 0cm

\newtheorem{theorem}{Theorem}[section]
\newtheorem{lemma}[theorem]{Lemma}

\newtheorem{definition}[theorem]{Definition}
\newtheorem{corollary}[theorem]{Corollary}

\title{No mixed graph with the nullity $\eta(\widetilde{G})=|V(G)|-2m(G)+2c(G)-1$}
\author{Shengjie  He$^{\rm a}$\footnote{Corresponding author.
Emails: he1046436120@126.com (Shengjie  He), rxhao@bjtu.edu.cn (Rong-Xia Hao), hjlai@math.wvu.edu (Hong-Jian Lai), lxygqzh@tjcu.edu.cn (Qiaozhi Geng)}, Rong-Xia Hao$^{\rm b}$, Hong-Jian Lai$^{\rm c}$, Qiaozhi Geng$^{\rm a}$\\
{\small\em $^{\rm a}$School of Science, Tianjin University of Commerce, Tianjin, 300134, China}\\
{\small\em $^{\rm b}$Department of Mathematics, Beijing Jiaotong University, Beijing,
100044, China}\\
{\small\em $^{\rm c}$Department of Mathematics, West Virginia University, Morgantown, WV 26506, USA}\\
}

\begin{document}
\baselineskip 0.55cm \maketitle

\begin{abstract}
A mixed graph $\widetilde{G}$ is obtained from a simple undirected graph $G$, the underlying graph of $\widetilde{G}$,
by orienting some edges of $G$. Let $c(G)=|E(G)|-|V(G)|+\omega(G)$ be the cyclomatic number of $G$ with $\omega(G)$ the
number of connected components of $G$, $m(G)$ be the matching number of $G$, and $\eta(\widetilde{G})$ be
the nullity of $\widetilde{G}$. Chen et al. (2018)\cite{LSC} and Tian et al. (2018)\cite{TFL}  proved independently that
$|V(G)|-2m(G)-2c(G) \leq  \eta(\widetilde{G}) \leq |V(G)|-2m(G)+2c(G)$,  respectively, and they characterized
the mixed graphs with nullity attaining the upper bound and the lower bound.
In this paper,  we prove that there is no mixed graph with nullity $\eta(\widetilde{G})=|V(G)|-2m(G)+2c(G)-1$. Moreover, for fixed $c(G)$, there are infinitely many connected mixed graphs with nullity
$|V(G)|-2m(G)+2c(G)-s$ $( 0 \leq s \leq 3c(G), s\neq1 )$ is proved.
\end{abstract}



\section{Introduction}

In this paper, we consider only graphs without multiedges and  loops. A simple $undirected$ $graph$ $G$ is denoted by $G=(V(G), E(G))$, where $V(G)$ is the vertex set and $E(G)$ is the edge set.
A $mixed$ $graph$ $\widetilde{G}$ is obtained by orienting some edges of $G$, where
$G$ is the underlying graph of $\widetilde{G}$.
The $Hermitian$-$adjacency$ $matrix$ of a mixed graph $\widetilde{G}$ of order $n$ is the $n \times n$ matrix $H(\widetilde{G}) = (h_{kl})$, where $h_{kl}=-h_{lk}=\mathbf{i}$ if
there is a directed edge from $v_{k}$ to $v_{l}$, where $\mathbf{i}$ is the imaginary number unit and $h_{kl}=h_{lk}=1$ if $v_{k}$ is connected to $v_{l}$ by an undirected edge, and $h_{kl}=0$ otherwise.
It is easy to see that $H(\widetilde{G})$ is a Hermitian  matrix, i.e., its conjugation and transposition is itself, that  is $H=H^{\ast}:=\overline{H}^{T}$. Thus
all its eigenvalues are real. The $positive$ $inertia$ $index$ (resp. the $negative$ $inertia$ $index$) of a mixed graph $\widetilde{G}$, denoted by $p^{+}(\widetilde{G})$ (resp. $n^{-}(\widetilde{G})$), is defined to be the number of positive eigenvalues (resp. negative eigenvalues)
of $H(\widetilde{G})$. The $rank$ of a mixed graph $\widetilde{G}$, denoted by $r(\widetilde{G})$,
is exactly the sum of $p^{+}(\widetilde{G})$ and $n^{-}(\widetilde{G})$. The $nullity$ of a mixed graph $\widetilde{G}$, denoted by $\eta(\widetilde{G})$, the algebraic multiplicity of the zero eigenvalues of $H(\widetilde{G})$. It is obviously that $\eta(\widetilde{G})=n-p^{+}(\widetilde{G})-n^{-}(\widetilde{G})$, where $n$ is the order of $\widetilde{G}$.
For a mixed cycle $\widetilde{C}$ of a mixed graph $\widetilde{G}$, the $signature$ of $\widetilde{C}$, denoted by $\sigma(\widetilde{C})$, is defined as $|f-b|$, where $f$ denotes the number of forward-oriented edges and $b$ denotes the number of backward-oriented edges of $\widetilde{C}$ without mentioning any direction.
Denote by $\widetilde{P_n}$, $\widetilde{S_n}$ and $\widetilde{C_n}$ a mixed path, mixed star and mixed cycle on $n$ vertices, respectively.
We refer to \cite{MATCH} for terminologies and notations undefined here.

For any vertex $v \in V(\widetilde{G})$, let $d_{\widetilde{G}}(v)$ (or simply $d(v)$)
denote the {\it degree} of $v$ in $\widetilde{G}$.
A vertex $v$ in $\widetilde{G}$ is called a {\it pendant vertex}
if $d(v)=1$,
and a vertex $u$ is called a {\it quasi-pendant vertex} of $v$ if $d(u)\ge 2$ and $u$ is adjacent to the pendant vertex $v$.
An $induced$ $subgraph$ $\widetilde{H}$ of $\widetilde{G}$ is a mixed graph such that the underlying graph of  $\widetilde{H}$ is an induced subgraph of the underlying graph of $\widetilde{G}$ and each edge of $\widetilde{H}$ has the same orientation (or non-orientation) as that in $\widetilde{G}$. For
$X \subseteq V(\widetilde{G})$, $\widetilde{G}-X$ is the mixed subgraph obtained from $\widetilde{G}$ by deleting all vertices in $X$ and
all incident edges. In particular, $\widetilde{G}-\{ x \}$ is usually written as $\widetilde{G}-x$ for simplicity.
For the sake of clarity, we use the notation $\widetilde{G}-\widetilde{H}$ instead of $\widetilde{G}-V(\widetilde{H})$ if $\widetilde{H}$ is an induced
subgraph of $\widetilde{G}$.

The $girth$ of the graph $G$ is the length of a shortest cycle in $G$, denoted by $g(G)$.
For an undirected $G$, the value $c(G)=|E(G)|-|V(G)|+\omega(G)$ is called the {\it cyclomatic number} of $G$,
where $\omega(G)$ is the number of connected components of $G$.
A set of pairwise independent edges of $G$ is called a $matching$, while a matching with the maximum
cardinality is a $maximum$ $matching$ of $G$. The $matching$ $number$ of $G$, denoted by $m(G)$,
is the cardinality of a maximum matching of $G$.
For a mixed graph $\widetilde{G}$, the cyclomatic number, denoted by $c(\widetilde{G})$,  and matching number, denoted by $m(\widetilde{G})$, are defined to be the cyclomatic number and
matching number of its underlying graph, respectively.
If any two cycles (if any) of $G$ share no common vertices, contracting each cycle of the graph $G$ into a vertex (called $cyclic$ $vertex$), we obtain a forest denoted by $T_{G}$.
Let $W_{G}$ be the vertex set consisting of all cyclic vertices. Denote by $[T_{G}]$ the subgraph of $T_{G}$ induced by all non-cyclic vertices.
Let $M$ be a matching a graph $G$, a vertex $u$ is said to be $M$-$saturated$ if some edge of $M$ is incident with $u$; otherwise, $u$ is said to be $M$-$unsaturated$.
A path $P$ is called an $M$-$alternating$ $path$ in $G$ if the edges of $P$ are alternately in $E(G) \setminus M$ (the set of edges belong to $E(G)$, but not $M$) and $M$. An $M$-$augmenting$ $path$ is an $M$-alternating path whose origin and terminus are $M$-unsaturated.

\begin{center}   \setlength{\unitlength}{0.7mm}
\begin{picture}(30,60)

\put(-55,30){\circle*{2}}
\put(-45,30){\circle*{2}}
\put(-43,28.5){$\cdots$}
\put(-35,30){\circle*{2}}
\put(-25,30){\circle*{2}}

\put(-55,30){\line(1,0){10}}
\put(-25,30){\line(-1,0){10}}

\put(-65,30){\circle{60}}
\put(-15,30){\circle{60}}

\put(-68,28.5){$C_{p}$}
\put(-18,28.5){$C_{q}$}

\put(-43,35){$P_{l}$}

\put(20,30){\circle*{2}}
\put(50,30){\circle*{2}}
\put(60,30){\circle*{2}}
\put(90,30){\circle*{2}}

\put(30,20){\circle*{2}}
\put(30,40){\circle*{2}}
\put(50,20){\circle*{2}}
\put(50,40){\circle*{2}}
\put(60,20){\circle*{2}}
\put(60,40){\circle*{2}}
\put(80,20){\circle*{2}}
\put(80,40){\circle*{2}}

\put(20,30){\line(1,1){10}}
\put(20,30){\line(1,-1){10}}

\put(90,30){\line(-1,1){10}}
\put(90,30){\line(-1,-1){10}}

\put(30,20){\line(1,0){20}}
\put(30,40){\line(1,0){20}}

\put(80,20){\line(-1,0){20}}
\put(80,40){\line(-1,0){20}}

\put(20,30){\line(1,0){30}}
\put(90,30){\line(-1,0){30}}

\put(52,28.5){$\cdots$}
\put(52,18.5){$\cdots$}
\put(52,38.5){$\cdots$}

\put(51.5,42.5){$P_{p+2}$}
\put(51.5,32.5){$P_{q+2}$}
\put(51.5,22){$P_{l+2}$}

\put(44,10){$\theta$-$(p, q, l)$}

\put(-51,10){$\infty$-$(p, q, l)$}

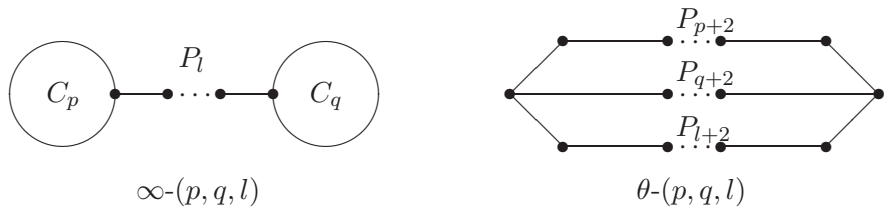
\put(-42,0){Fig. 1. The graphs $\infty$-$(p, q, l)$ and $\theta$-$(p, q, l)$}

\end{picture} \end{center}

There are two basic bicyclic graphs \cite{BIC}: $\infty$-graph and $\theta$-graph, which are depicted in Fig. 1. An $\infty$-graph, denoted by
$\infty$-$(p, l, q)$, is obtained from two vertex-disjoint cycles $C_{p}$ and $C_{q}$ by connecting some vertex of $C_{p}$ and some vertex of $C_{q}$
with a path of length $l-1$ (in the case of $l=1$, identifying the  two vertices mentioned above); and a $\theta$-graph, denoted by $\theta$-$(p, l, q)$,
is a union of three internally disjoint paths $P_{p+2}$, $P_{l+2}$, $P_{q+2}$ with common end vertices. It can be checked that any bicyclic graph can be obtained from an $\infty$-graph or a $\theta$-graph by attaching
some trees to some of its vertices.

\begin{center}   \setlength{\unitlength}{0.7mm}
\begin{picture}(100,190)

\put(-55,30){\circle*{2}}
\put(-40,40){\circle*{2}}
\put(-30,30){\circle*{2}}
\put(-40,20){\circle*{2}}
\put(-20,40){\circle*{2}}
\put(-20,30){\circle*{2}}
\put(-20,20){\circle*{2}}
\put(-10,40){\circle*{2}}
\put(-10,20){\circle*{2}}

\put(0,40){\circle*{2}}
\put(-5,30){\circle*{2}}
\put(0,20){\circle*{2}}

\put(10,40){\circle*{2}}
\put(10,20){\circle*{2}}

\put(20,40){\circle*{2}}
\put(20,20){\circle*{2}}

\put(-55,30){\line(3,2){15}}
\put(-55,30){\line(3,-2){15}}
\put(-55,30){\line(1,0){25}}
\put(-40,40){\line(1,0){10}}
\put(-40,20){\line(1,0){10}}
\put(-30,20){\circle*{2}}
\put(-30,40){\circle*{2}}

\put(-28,28.5){$\cdots$}
\put(-28,38.5){$\cdots$}
\put(-28,18.5){$\cdots$}
\put(2,38.5){$\cdots$}
\put(2,18.5){$\cdots$}
\put(-20,30){\line(1,0){15}}
\put(-20,20){\line(1,0){20}}
\put(-20,40){\line(1,0){20}}
\put(-20,20){\line(3,2){15}}
\put(-20,40){\line(3,-2){15}}
\put(20,20){\line(-1,0){10}}
\put(20,40){\line(-1,0){10}}
\put(20,40){\line(0,-1){20}}
\put(50,30){\circle*{2}}

\put(60,40){\circle*{2}}
\put(60,20){\circle*{2}}

\put(70,40){\circle*{2}}
\put(70,30){\circle*{2}}
\put(80,40){\circle*{2}}
\put(80,30){\circle*{2}}
\put(90,40){\circle*{2}}
\put(90,20){\circle*{2}}
\put(100,30){\circle*{2}}

\put(110,40){\circle*{2}}
\put(110,20){\circle*{2}}

\put(120,40){\circle*{2}}
\put(120,30){\circle*{2}}
\put(130,40){\circle*{2}}
\put(130,30){\circle*{2}}
\put(140,40){\circle*{2}}
\put(140,20){\circle*{2}}
\put(150,30){\circle*{2}}
\put(150,30){\line(-1,1){10}}
\put(150,30){\line(-1,-1){10}}
\put(150,30){\line(-1,0){20}}

\put(50,30){\line(1,1){10}}
\put(50,30){\line(1,-1){10}}
\put(50,30){\line(1,0){20}}

\put(100,30){\line(-1,1){10}}
\put(100,30){\line(1,1){10}}
\put(100,30){\line(-1,0){20}}
\put(100,30){\line(1,0){20}}

\put(60,40){\line(1,0){10}}
\put(60,20){\line(1,0){30}}
\put(140,20){\line(-1,0){30}}

\put(140,40){\line(-1,0){10}}
\put(120,40){\line(-1,0){10}}
\put(90,40){\line(-1,0){10}}

\put(72,38.5){$\cdots$}
\put(72,28.5){$\cdots$}
\put(122,38.5){$\cdots$}
\put(122,28.5){$\cdots$}
\put(92.5,18.5){$\cdots$}
\put(100.5,18.5){$\cdots$}

\put(-20,70){\circle*{2}}
\put(-20,80){\circle*{2}}
\put(-20,60){\circle*{2}}
\put(-20,90){\circle*{2}}

\put(-10,70){\circle*{2}}
\put(-10,80){\circle*{2}}
\put(-10,60){\circle*{2}}
\put(-10,90){\circle*{2}}

\put(-40,70){\circle*{2}}
\put(-30,80){\circle*{2}}
\put(-30,60){\circle*{2}}
\put(-30,90){\circle*{2}}

\put(0,80){\circle*{2}}
\put(0,60){\circle*{2}}
\put(0,90){\circle*{2}}
\put(10,70){\circle*{2}}

\put(-40,70){\line(1,0){20}}
\put(-40,70){\line(1,-1){10}}
\put(-40,70){\line(1,2){10}}
\put(-40,70){\line(1,1){10}}

\put(10,70){\line(-1,0){20}}
\put(10,70){\line(-1,-1){10}}
\put(10,70){\line(-1,2){10}}
\put(10,70){\line(-1,1){10}}

\put(0,90){\line(-1,0){10}}
\put(0,80){\line(-1,0){10}}
\put(0,60){\line(-1,0){10}}
\put(-20,90){\line(-1,0){10}}
\put(-20,80){\line(-1,0){10}}
\put(-20,60){\line(-1,0){10}}

\put(-18.5,58.5){$\cdots$}
\put(-18.5,68.5){$\cdots$}
\put(-18.5,78.5){$\cdots$}
\put(-18.5,88.5){$\cdots$}

\put(95,80){\circle*{2}}
\put(95,60){\circle*{2}}
\put(95,90){\circle*{2}}
\put(95,70){\circle*{2}}
\put(105,80){\circle*{2}}
\put(105,60){\circle*{2}}
\put(105,90){\circle*{2}}
\put(105,70){\circle*{2}}

\put(85,60){\circle*{2}}
\put(115,60){\circle*{2}}
\put(85,80){\circle*{2}}
\put(115,80){\circle*{2}}
\put(75,70){\circle*{2}}
\put(125,70){\circle*{2}}
\put(85,90){\circle*{2}}
\put(115,90){\circle*{2}}

\put(75,70){\line(1,1){10}}
\put(75,70){\line(1,0){20}}
\put(75,70){\line(1,-1){10}}

\put(125,70){\line(-1,1){10}}
\put(125,70){\line(-1,0){20}}
\put(125,70){\line(-1,-1){10}}

\put(115,90){\line(-1,0){10}}
\put(115,90){\line(0,-1){10}}
\put(85,90){\line(1,0){10}}
\put(85,90){\line(0,-1){10}}

\put(95,80){\line(-1,0){10}}
\put(95,60){\line(-1,0){10}}
\put(105,80){\line(1,0){10}}
\put(105,60){\line(1,0){10}}
\put(96.5,58.5){$\cdots$}
\put(96.5,68.5){$\cdots$}
\put(96.5,78.5){$\cdots$}
\put(96.5,88.5){$\cdots$}

\put(-50,120){\circle*{2}}
\put(-40,110){\circle*{2}}
\put(-40,130){\circle*{2}}
\put(-30,120){\circle*{2}}
\put(-30,110){\circle*{2}}
\put(-30,130){\circle*{2}}
\put(-20,120){\circle*{2}}
\put(-20,110){\circle*{2}}
\put(-20,130){\circle*{2}}

\put(-10,110){\circle*{2}}
\put(-10,130){\circle*{2}}
\put(0,120){\circle*{2}}
\put(5,120){\circle*{2}}
\put(15,120){\circle*{2}}
\put(20,120){\circle*{2}}
\put(28,120){\circle{15}}

\put(-50,120){\line(1,0){20}}
\put(-50,120){\line(1,1){10}}
\put(-50,120){\line(1,-1){10}}

\put(-40,110){\line(1,0){10}}
\put(-40,130){\line(1,0){10}}

\put(-20,110){\line(1,0){10}}
\put(-20,130){\line(1,0){10}}

\put(0,120){\line(-1,0){20}}
\put(0,120){\line(-1,1){10}}
\put(0,120){\line(-1,-1){10}}

\put(0,120){\line(1,0){5}}
\put(15,120){\line(1,0){5}}
\put(-28.5,108.5){$\cdots$}
\put(-28.5,118.5){$\cdots$}
\put(-28.5,128.5){$\cdots$}
\put(6.5,118.5){$\cdots$}

\put(70,120){\circle*{2}}
\put(80,110){\circle*{2}}
\put(80,130){\circle*{2}}
\put(90,120){\circle*{2}}
\put(90,110){\circle*{2}}
\put(90,130){\circle*{2}}
\put(100,120){\circle*{2}}
\put(100,110){\circle*{2}}
\put(100,130){\circle*{2}}

\put(110,110){\circle*{2}}
\put(110,130){\circle*{2}}
\put(120,120){\circle*{2}}

\put(70,120){\line(1,0){20}}
\put(70,120){\line(1,1){10}}
\put(70,120){\line(1,-1){10}}

\put(80,110){\line(1,0){10}}
\put(80,130){\line(1,0){10}}

\put(100,110){\line(1,0){10}}
\put(100,130){\line(1,0){10}}

\put(120,120){\line(-1,0){20}}
\put(120,120){\line(-1,1){10}}
\put(120,120){\line(-1,-1){10}}

\put(91.5,108.5){$\cdots$}
\put(91.5,118.5){$\cdots$}
\put(91.5,128.5){$\cdots$}

\put(110,110){\circle*{2}}
\put(115,110){\circle*{2}}
\put(125,110){\circle*{2}}
\put(130,110){\circle*{2}}

\put(110,110){\line(1,0){5}}
\put(125,110){\line(1,0){5}}
\put(116.5,108.5){$\cdots$}
\put(138,110){\circle{15}}

\put(-10,155){\circle{15}}
\put(-17,155){\circle*{2}}
\put(-3,155){\circle*{2}}

\put(-22,155){\circle*{2}}
\put(2,155){\circle*{2}}
\put(-32,155){\circle*{2}}
\put(12,155){\circle*{2}}

\put(-37,155){\circle*{2}}
\put(17,155){\circle*{2}}
\put(24.7,155){\circle{15}}

\put(-43.9,155){\circle{15}}
\put(17,155){\line(-1,0){5}}
\put(2,155){\line(-1,0){5}}

\put(-17,155){\line(-1,0){5}}
\put(-37,155){\line(1,0){5}}

\put(-30.5,153.5){$\cdots$}
\put(4.5,153.5){$\cdots$}

\put(100,155){\circle*{2}}

\put(105,155){\circle*{2}}
\put(95,155){\circle*{2}}
\put(100,160){\circle*{2}}

\put(115,155){\circle*{2}}
\put(85,155){\circle*{2}}
\put(100,170){\circle*{2}}

\put(120,155){\circle*{2}}
\put(80,155){\circle*{2}}

\put(100,155){\line(1,0){5}}
\put(100,155){\line(-1,0){5}}
\put(100,155){\line(0,1){5}}

\put(120,155){\line(-1,0){5}}
\put(80,155){\line(1,0){5}}

\put(107,153.5){$\cdots$}
\put(87,153.5){$\cdots$}

\put(73,155){\circle{15}}
\put(127,155){\circle{15}}
\put(100,177){\circle{15}}
\put(99.3,162.2){$\vdots$}

\put(98,13){$T_{8}$}
\put(98,53){$T_{6}$}
\put(98,103){$T_{4}$}
\put(98,146){$T_{2}$}

\put(-17,13){$T_{7}$}
\put(-17,53){$T_{5}$}
\put(-17,103){$T_{3}$}
\put(-12,142){$T_{1}$}

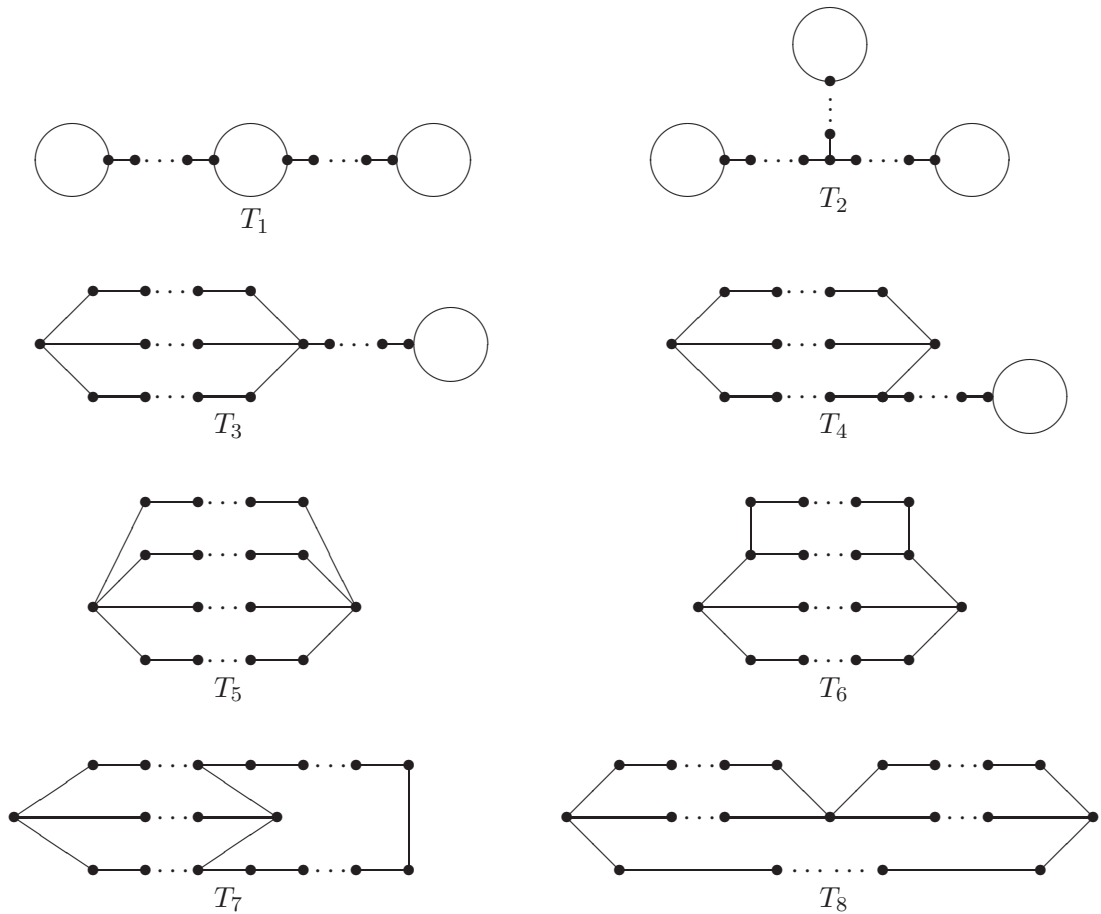
\put(2,0){Fig. 2. The basic tricyclic graphs $T_{1}-T_{8}$}

\end{picture} \end{center}

A $base$ of a tricyclic graph $G$ is a minimal tricyclic subgraph (i.e., containing no pendant vertex) of $G$. From \cite{TRIC}, there are eight types of bases for tricyclic graph, which are depicted in Fig. 2.
Note that any tricyclic graph $G$ can be obtained from the base of $G$ by attaching trees to some vertices of the base of $G$.

In recent years, the study on the Hermitian adjacent matrix of mixed graphs received increasing attention. In \cite{LXL}, Liu and Li investigated the properties of the coefficients of characteristic polynomials of
mixed graphs and cospectral problems among mixed graphs.
Guo and Mohar \cite{MOHAR} presented some basic properties of the rank of the mixed graphs, and many differences from the properties of eigenvalues of undirected graphs were discussed.
In \cite{WANGLONG}, Wang et al. researched the the relation among the rank, the matching number and the cyclomatic number of an undirected graph and obtained that $2m(G)-2c(G) \leq r(G) \leq 2m(G)+c(G)$.
The undirected graphs with $\eta(G)= |V(G)|-2m(G)-c(G)$ was characterized by Wang \cite{WANGL} and the undirected graphs with $\eta(G)= |V(G)|-2m(G)+2c(G)$ was characterized by Song et al. \cite{SONG}.
Chen et al. \cite{LSC} and Tian et al. \cite{TFL} studied  independently the lower and upper bounds of the rank of the mixed graphs in terms of the matching number, and the mixed graphs with rank attaining the upper bound and the lower bound were characterized, respectively.
For other related research of the adjacent matrix of a graph, one may be referred to those in \cite{GUT,HLSC,WDY,MAH2,WONGEUR,BEVI,HSJ}.

The study on the mixed graphs with fixed nullity has been a popular subject in the graph theory.
Mohar \cite{MOHAR2} characterized all the mixed graphs with rank equal to 2. Wang et al. \cite{YBJ} studied the graphs with $H$-rank 3.
Yang et al. \cite{YJL} characterize all connected mixed graphs with $H$-rank 4 (resp., 6 or 8) among all mixed graphs containing induced mixed odd cycles whose lengths are no less than 5 (resp., 7 or 9).
Li and Guo \cite{LG} proved that there is no graph with nullity $\eta(G) = |V (G)|-2m(G)+2c(G)-1$, and for fixed $c(G)$, infinitely many connected graphs with nullity $\eta(G) = |V (G)|-2m(G)+2c(G)-s$,
where $0  \leq  s  \leq 3c(G), s \neq 1$ are also constructed.
Lu and Wu \cite{LUWU} proved that there is no signed graph with nullity $\eta(G, \sigma) = |V (G)|- 2m(G)+2c(G)-1$, and for fixed $c(G)$, infinitely many connected signed graphs with nullity $\eta(G, \sigma) = |V (G)|-2m(G)+ 2c(G)-s$, where $0  \leq  s  \leq 3c(G), s \neq 1$ are also constructed.
In this paper, we prove that no mixed graph with nullity $\eta(\widetilde{G}) = |V (G)|- 2m(G)+2c(G)-1$, and for fixed $c(G)$, there are infinitely many connected mixed graphs with nullity $\eta(\widetilde{G}) = |V (G)|-2m(G)+ 2c(G)-s$, where $0  \leq  s  \leq 3c(G)$ and $s \neq 1$.

Our main results are the following Theorems \ref{T1} and \ref{T2}.

\begin{theorem}\label{T1}
Let $\widetilde{G}$ be a mixed graph. Then
$$\eta(\widetilde{G}) \neq |V(G)|-2m(G)+2c(G)-1.$$
\end{theorem}

\begin{theorem}\label{T2}
For a fixed value $c(G)$, there are infinitely many connected mixed graphs with the nullity $\eta(\widetilde{G})=|V(G)|-2m(G)+2c(G)-k$, where $0 \leq k \leq 3c(G)$ and $k \neq 1$.
\end{theorem}

The rest of this paper is organized as follows. In Section 2, some useful lemmas are listed which will be used in the proof of our main results.
The proof of the Theorem \ref{T1} is presented in Section 3. In Section 4, the proof for Theorem \ref{T2} is given.

\section{Preliminaries}

We need the following known results and useful lemmas to prove our main results, which will be needed in next sections.

\begin{lemma} \label{L16}{\rm\cite{MOHAR2}}
Let $\widetilde{G}$ be a mixed graph.

{\em(i)} If $\widetilde{H}$ is an induced subgraph of $\widetilde{G}$, then $r(\widetilde{H}) \leq r(\widetilde{G})$.

{\em(ii)} If $\widetilde{G_{1}}, \widetilde{G_{2}}, \cdots, \widetilde{G_{t}}$ are the connected components of $\widetilde{G}$, then $r(\widetilde{G})=\sum_{i=1}^{t}r(\widetilde{G_{t}})$.

{\em(iii)}  $r(\widetilde{G})\geq 0$, where equality if and only if $\widetilde{G}$ is an empty graph.

\end{lemma}

\begin{lemma} \label{L131}{\rm\cite{YBJ}}
Let $\widetilde{T}$ be a mixed tree. Then $r(\widetilde{T})=2m(T)$.
\end{lemma}

\begin{lemma} \label{L17}{\rm\cite{LSC}}
Let $G$ be a simple undirected graph.
Then $m(G)-1 \leq m(G-v) \leq m(G)$ for any vertex $v \in V(G)$.
\end{lemma}


\begin{lemma} \label{L19}{\rm\cite{WDY}}
Let $x$ be a pendant vertex of $G$ and $y$ be the neighbour of $x$. Then $m(G)=m(G-y)+1=m(G- \{ x, y \})+1$.
\end{lemma}

\begin{lemma}\label{MATCH}{\rm\cite{MATCH}}
A matching $M$ of graph $G$ is a maximum matching if and only if $G$ contains no $M$-augmenting path.
\end{lemma}

\begin{lemma} \label{L13}{\rm\cite{YBJ}}
Let $x$ be a pendant vertex of $\widetilde{G}$ and $y$ be the neighbour of $x$. Then
$$\eta(\widetilde{G})=\eta(\widetilde{G}-x-y).$$
\end{lemma}


\begin{lemma} \label{L0001}{\rm\cite{MOHAR2}}
Let $x$ be a vertex of a mixed graph $\widetilde{G}$.
Then $\eta(\widetilde{G})-1 \leq \eta (\widetilde{G}-x) \leq \eta(\widetilde{G})+1$.
\end{lemma}


\begin{lemma} \label{L23}{\rm\cite{WDY}}
Let $G$ be a graph with $x \in V(G)$. Then

{\em(i)} $c(G)=c(G-x)$ if $x$ is not lying in a cycle of $G$;

{\em(ii)} $c(G-x) \leq c(G)-1$ if $x$ lies in a cycle of $G$;

{\em(iii)} $c(G-x) \leq c(G)-2$ if $x$ is a common vertex of distinct cycles of $G$.
\end{lemma}

\begin{theorem}[\cite{LSC,TFL}]\label{mixed2}
Let $\widetilde{G}$ be a connected mixed graph. Then
$$
|V(G)|-2m(G)-c(G) \leq \eta(\widetilde{G}) \leq |V(G)|-2m(G)+2c(G).
$$
\end{theorem}

\begin{theorem}[\cite{LSC,TFL}]\label{mixed1}
Let $\widetilde{G}$ be a connected mixed graph.
Then $\eta(\widetilde{G})=|V(G)|-2m(G)+2c(G)$ if and only if  all the following conditions hold for $\widetilde{G}$:

{\em(i)} the cycles (if any) of $\widetilde{G}$ are pairwise vertex-disjoint;

{\em(ii)} each cycle $\widetilde{C_{l}}$ of $\widetilde{G}$ is even with $\sigma(\widetilde{C_{l}}) \equiv   l \, (\rm{mod} \ 4)$;

{\em(iii)} $m(T_{G})=m(G-O(G))$, where $O(G)$ is the set of vertices in cycles of $\widetilde{G}$.
\end{theorem}

\begin{lemma} \label{L25}{\rm\cite{YBJ}} Let $\widetilde{C_{n}}$ be a mixed cycle with $n$ vertices. Then
$$r(\widetilde{C_{n}})=\left\{
             \begin{array}{ll}
               n-1, & \hbox{if $n$ is odd, $\sigma(\widetilde{C_{n}})$ is odd;} \\
               n, & \hbox{if $n$ is odd, $\sigma(\widetilde{C_{n}})$ is even;} \\
               n, & \hbox{if $n$ is even, $\sigma(\widetilde{C_{n}})$ is odd;} \\
               n, & \hbox{if $n$ is even, $n+\sigma(\widetilde{C_{n}}) \equiv 2 \ (\rm{mod} \ 4)$;} \\
               n-2, & \hbox{if $n$ is even, $n+\sigma(\widetilde{C_{n}}) \equiv 0 \ (\rm{mod} \ 4)$.}
             \end{array}
           \right.
$$
\end{lemma}

\begin{theorem}\label{UNMIX}{\rm\cite{HSJ}}
Let $\widetilde{G}$ be a  mixed unicyclic graph with the cycle $\widetilde{C_{q}}$. Then we have
$$  (p^{+}(\widetilde{G}), n^{-}(\widetilde{G})) = \left\{
  \begin{array}{ll}
      (m(\widetilde{G})-1, m(\widetilde{G})-1), & \hbox{ if $q$ and $\sigma(\widetilde{C_{q}})$ are even, $q-\sigma(\widetilde{C_{q}}) \equiv 0 \ (\rm{mod} \ 4)  $ } \\
              & \hbox{ and no maximum matching contains an edge } \\
              & \hbox{ incident to the cycle;} \\
      (m(\widetilde{G})+1, m(\widetilde{G})), & \hbox{ if $q$ is odd, $\sigma(\widetilde{C_{q}})$ is even, $q-\sigma(\widetilde{C_{q}}) \equiv 1 \ (\rm{mod} \ 4)  $ } \\
                 & \hbox{ and $m(\widetilde{G})=m(\widetilde{G}-V(\widetilde{C_{q}}))+\frac{q-1}{2}$ ;} \\
      (m(\widetilde{G}), m(\widetilde{G})+1), & \hbox{ if $q$ is odd, $\sigma(\widetilde{C_{q}})$ is even, $q-\sigma(\widetilde{C_{q}}) \equiv 3 \ (\rm{mod} \ 4)  $ } \\
                      & \hbox{ and $m(\widetilde{G})=m(\widetilde{G}-V(\widetilde{C_{q}}))+\frac{q-1}{2}$ ;} \\
      (m(\widetilde{G}), m(\widetilde{G})), & \hbox{otherwise.}
    \end{array}
  \right.
$$
\end{theorem}

From Lemma \ref{UNMIX}, the following Corollary \ref{UNM} can be obtained immediately.

\begin{corollary}\label{UNM}
Let $\widetilde{H}$ be a mixed unicyclic graph. Then
$$\eta(\widetilde{H}) \neq |V(H)|-2m(H)+2c(H)-1.$$
\end{corollary}

\section{Proof of Theorem \ref{T1}.}

In this section, the proof for Theorem \ref{T1} is provided. Firstly, an operation on graphs is introduced.

\begin{definition}\label{PED}
Let $G$ be a graph with at least one pendant vertex. The operation of deleting a pendant vertex and its adjacent
vertex from $G$ is called PED (short form for `pendant edge deletion').
\end{definition}

In the following, some useful lemmas which
will be used to prove the main result of this section are introduced.

Let $D$ be the graph in Fig. 3. It is easy to see that $m(D)=2$ and $c(D)=2$.

\begin{center}   \setlength{\unitlength}{0.7mm}
\begin{picture}(30,60)

\put(15,10){\circle*{2}}
\put(15,30){\circle*{2}}
\put(15,50){\circle*{2}}
\put(-25,30){\circle*{2}}
\put(55,30){\circle*{2}}

\put(15,30){\line(1,0){40}}
\put(15,30){\line(-1,0){40}}

\put(-25,30){\line(2,-1){40}}
\put(-25,30){\line(2,1){40}}

\put(55,30){\line(-2,1){40}}
\put(55,30){\line(-2,-1){40}}

\put(-12,-5){Fig. 3. The graph $D$}

\end{picture} \end{center}

\begin{lemma}\label{L880}
Let $\widetilde{G}$ be a mixed graph with the underlying graph $G \cong D$. Then $r(\widetilde{G}) \geq 2$, i.e.,
$$\eta(\widetilde{G})\neq |V(G)|- 2m(G)+2c(G)-s, s=0, 1.$$

\end{lemma}
\begin{proof} By definition, the adjacency matrix of $\widetilde{G}$ can be written by
$$H(\widetilde{G})=\left(
  \begin{array}{ccccc}
    0 & \alpha_{1} & \alpha_{2} & \alpha_{3} & 0 \\
    \alpha_{4} & 0 & 0 & 0 & \alpha_{5} \\
    \alpha_{6} & 0 & 0 & 0 & \alpha_{7} \\
    \alpha_{8} & 0 & 0 & 0 & \alpha_{9} \\
    0 & \alpha_{10} & \alpha_{11} & \alpha_{12} & 0 \\
  \end{array}
\right),$$
where $\alpha_{i} \in \{ 1, \mathbf{i}, \mathbf{-i}  \} $ for $1 \leq i \leq 12$. It can be checked that the vectors $(0, \alpha_{1}, \alpha_{2}, \alpha_{3}, 0)$ and
$(\alpha_{4}, 0, 0, 0, \alpha_{5})$ are linearly independent. Thus, $r(\widetilde{G}) \geq 2$ and the result follows immediately.
\end{proof}




\begin{lemma}\label{L771}
Let $x$ be a pendant vertex of a mixed $\widetilde{G}$, and the quasi-pendant vertex $y$ of $x$ does not lie in any cycle of $\widetilde{G}$. If $\eta(\widetilde{G}-x-y) \neq |V(G-x-y)|-2m(G-x-y)+2c(G-x-y)-s$ $(s=0, 1)$, then
$$\eta(\widetilde{G}) \neq |V(G)|-2m(G)+2c(G)-s.$$
\end{lemma}
\begin{proof}
Since $y$ does not lie in any cycle of $\widetilde{G}$, by lemmas \ref{L19} and \ref{L23}, one has
$$m(G-x-y)=m(G)-1,$$
$$c(G-x-y)=c(G),$$
and
$$|V(G-x-y)|=|V(G)|-2.$$
Then, by Lemma \ref{L13}, we have
\begin{eqnarray*}
\eta(\widetilde{G})
&=&\eta(\widetilde{G}-x-y)\\
&\neq&|V(G-x-y)|-2m(G-x-y)+2c(G-x-y)-s\\
&=&|V(G)|-2-2[m(G)-1]+2c(G)-s\\
&=&|V(G)|-2m(G)+2c(G)-s.\\
\end{eqnarray*}

This completes the proof.
\end{proof}

\begin{lemma}\label{L881}
Let $x$ be a pendant vertex of a mixed $\widetilde{G}$ and $y$ be the quasi-pendant vertex of $x$. If $y$ does not lie in any cycle of $\widetilde{G}$,
then $\eta(\widetilde{G})=|V(G)|-2m(G)+2c(G)-s$  $(0 \leq s \leq 3c(G))$ if and only if
$\eta(\widetilde{G}-x-y)=|V(G-x-y)|-2m(G-x-y)+2c(G-x-y)-s$.
\end{lemma}
\begin{proof}

(Sufficiency.) From Lemmas \ref{L19} and \ref{L23}, one has
$$m(G-x-y)=m(G)-1,$$
$$c(G-x-y)=c(G),$$
and
$$|V(G-x-y)|=|V(G)|-2.$$

By Lemma \ref{L13} and $\eta(\widetilde{G}-x-y)=|V(G-x-y)|-2m(G-x-y)+2c(G-x-y)-s$, we have
\begin{eqnarray*}
\eta(\widetilde{G})
&=&\eta(\widetilde{G}-x-y)\\
&=&|V(G-x-y)|-2m(G-x-y)+2c(G-x-y)-s\\
&=&|V(G)|-2-2[m(G)-1]+2c(G)-s\\
&=&|V(G)|-2m(G)+2c(G)-s.\\
\end{eqnarray*}

(Necessity.) By Lemma \ref{L13}, one has that
\begin{eqnarray*}
\eta(\widetilde{G}-x-y)
&=&\eta(\widetilde{G})\\
&=&|V(G)|-2m(G)+2c(G)-s\\
&=&|V(G-x-y)|+2-2[m(G-x-y)+1]+2c(G-x-y)-s\\
&=&|V(G-x-y)|-2m(G-x-y)+2c(G-x-y)-s.\\
\end{eqnarray*}

This completes the proof.
\end{proof}

\begin{lemma}\label{L883}
Let $\widetilde{G}$ be a mixed graph with a pendant vertex $x$, and $y$ be the quasi-pendant vertex of $x$. If $y$ lies in some cycle of $\widetilde{G}$, then
$$\eta(\widetilde{G}) \leq |V(G)|-2m(G)+2c(G)-2.$$
\end{lemma}
\begin{proof} Suppose to the contrary that $\eta(\widetilde{G}) > |V(G)|-2m(G)+2c(G)-2$. From Lemma \ref{mixed2}, one has that
$$\eta(\widetilde{G})= |V(G)|-2m(G)+2c(G)$$
or
$$\eta(\widetilde{G})= |V(G)|-2m(G)+2c(G)-1. $$

Since $y$ lies in some cycle of $\widetilde{G}$, by Lemmas \ref{L19} and \ref{L23}, we have
$$|V(G-x-y)|=|V(G)|-2,$$
$$m(G-x-y) =m(G)-1,$$
and
$$c(G-x-y) \leq c(G)-1.$$
By Lemma \ref{L13}, one has that
\begin{eqnarray*}
\eta(\widetilde{G}-x-y)
&=&\eta(\widetilde{G})\\
&\geq&|V(G)|-2m(G)+2c(G)-1\\
&\geq&|V(G-x-y)|+2-2[m(G-x-y)+1]+2c(G-x-y)+2-1\\
&\geq&|V(G-x-y)|-2m(G-x-y)+2c(G-x-y)+1.
\end{eqnarray*}
Which contradicts to Lemma \ref{mixed2}. Thus, for any mixed graph $\widetilde{G}$ with a quasi-pendant vertex $y$ lies in some cycle of $\widetilde{G}$, one has
$$\eta(\widetilde{G}) \leq |V(G)|-2m(G)+2c(G)-2.$$

This completes the proof.
\end{proof}

\begin{lemma}\label{L884}
Let $\widetilde{G}$ be a mixed graph without pendant vertices and $G \neq D$ (where $D$ is shown in Fig. 3.). If
$\eta(\widetilde{G}) \neq |V(G)|-2m(G)+2c(G)$ and $c(G) \geq 2$, then there exists a vertex $u$ in some cycle of $\widetilde{G}$ such that
$\eta(\widetilde{G}-u) \neq |V(G-u)|-2m(G-u)+2c(G-u)$.
\end{lemma}
\begin{proof}  We will deal with the problem with two cases according to $g(G)$.

{\bf  Case 1.} $g(G) =3$.

Since $g(G) =3$, there exists a cycle of $\widetilde{G}$, denoted by $\widetilde{C_{q}}$, with length three.
Since $c(G) \geq 2$ and no pendant vertices in $\widetilde{G}$, there exists a vertex $u$ on some cycle in $\widetilde{G}$ such that $\widetilde{C_{q}}$ is mixed cycle of $\widetilde{G}-u$.
Which implies that $g(G-u)=3$, i.e., $\widetilde{G}-u$ does not satisfy the condition of Lemma \ref{mixed1}(ii). Thus, by Lemma \ref{mixed1}, one has
$$\eta(\widetilde{G}-u) \neq |V(G-u)|-2m(G-u)+2c(G-u).$$

{\bf  Case 2.} $g(G)  \geq 4$.

By Lemma \ref{mixed1} and $\eta(\widetilde{G}) \neq |V(G)|-2m(G)+2c(G)$, one has that $\widetilde{G}$ does not satisfy at least one of the three conditions in Lemma \ref{mixed1}.

{\bf  Subcase 2.1.} $\widetilde{G}$ does not satisfy Lemma \ref{mixed1}(i).

Note that $\widetilde{G}$ contains at least two vertex-joint cycles, denoted by $\widetilde{C_{k}}$ and $\widetilde{C_{s}}$ $(k, s \geq 4)$.
Let $\widetilde{G}[\widetilde{C_{k}}, \widetilde{C_{s}}]$ be the mixed graph induced by $\widetilde{C_{k}}$ and $\widetilde{C_{s}}$.

{\bf  Subcase 2.1.1.} $c(G)=2$.

Since $\widetilde{G}$ contains no pendant vertices, $\widetilde{G}$ is the union of an $\infty$-graph (or a $\theta$-graph) and some isolated vertices (if any), which implies that
$\widetilde{G}[\widetilde{C_{k}}, \widetilde{C_{s}}]$ is either $\infty$-$\widetilde{(p, l, q)}$ or $\theta$-$\widetilde{(p, l, q)}$.
Since $G \neq D$, $\widetilde{G}[\widetilde{C_{k}}, \widetilde{C_{s}}]$ is not $\theta$-$\widetilde{(1, 1, 1)}$.
As shown in Fig. 1., it can be checked that there exists a vertex $u$ in some cycle of $\widetilde{G}$ such that either $\widetilde{G}-u$ contains a quasi-pendant vertex lies in some cycle of $\widetilde{G}$ or
there exists a integer $t$ $(t \geq 0)$ such that $\widetilde{G_{t}}$ contains a quasi-pendant vertex lies in some cycle of $\widetilde{G}$
(Where $\widetilde{G_{0}}$ is obtained by deleting a pendant vertex and its quasi-pendant vertex of $\widetilde{G}-u$.
If $\widetilde{G_{0}}$ contains a quasi-pendant vertex lies in some cycle of $\widetilde{G}$, then $\widetilde{G_{0}}$
is as we required and we are done. Otherwise, a subgraph $\widetilde{G_{1}}$ of $\widetilde{G_{0}}$ can be obtained after deleting a pendant vertex and its quasi-pendant vertex of $\widetilde{G_{0}}$.
If $\widetilde{G_{1}}$ contains a quasi-pendant vertex lies in some cycle of $\widetilde{G}$,
then $\widetilde{G_{1}}$ is as we required and we are done. Otherwise, repeating the above steps until we obtain a mixed graph $\widetilde{G_{t}}$ such that $\widetilde{G_{t}}$ contains
a quasi-pendant vertex lies in some cycle of $\widetilde{G}$.).


If there exists a vertex $u$ in some cycle of $\widetilde{G}$ such that $\widetilde{G}-u$ contains a quasi-pendant vertex lies in some cycle of $\widetilde{G}$,
by Lemma \ref{L883}, one has
$$\eta(\widetilde{G}-u) \neq |V(G-u)|-2m(G-u)+2c(G-u).$$

If $\widetilde{G_{t}}$ contains a quasi-pendant vertex lies in some cycle of $\widetilde{G}$,
by Lemma \ref{L883}, we have
$$\eta(\widetilde{G_{t}}) \neq |V(G_{t})|-2m(G_{t})+2c(G_{t}).$$
Since $\widetilde{G_{t}}$ is obtained by deleting a series of pendant vertices and their quasi-pendant vertices of $\widetilde{G}-u$ and these quasi-pendant vertices lie in no cycle of $\widetilde{G}$,
by Lemma \ref{L771}, we have
$$\eta(\widetilde{G}-u) \neq |V(G-u)|-2m(G-u)+2c(G-u).$$

{\bf  Subcase 2.1.2.} $c(G) \geq 3$.

If there exists a vertex $u$ in the cycle of $\widetilde{G}$ such that $u \notin \widetilde{G}[\widetilde{C_{k}}, \widetilde{C_{s}}]$. Which implies that
$\widetilde{G}[\widetilde{C_{k}}, \widetilde{C_{s}}]$ is a subgraph of $\widetilde{G}-u$. Then $\widetilde{G}-u$ does not satisfy Lemma \ref{mixed1}(i), hence
$$\eta(\widetilde{G}-u) \neq |V(G-u)|-2m(G-u)+2c(G-u).$$
For example, as shown in Fig. 2., the mixed graphs with $T_{i}$ $(i=1, 2, 3, 4)$ as underlying graph, which contains a vertex $u$ on the cycle
and $u \notin \widetilde{G}[\widetilde{C_{k}}, \widetilde{C_{s}}]$.

Next, one can suppose $u \in \widetilde{G}[\widetilde{C_{k}}, \widetilde{C_{s}}]$ for a vertex $u$ in any cycle. It implies that for each vertex $u \notin \widetilde{G}[\widetilde{C_{k}}, \widetilde{C_{s}}]$,
$u$ is not in any cycle. That is any cycle of $\widetilde{G}$ is
the subgraph of $\widetilde{G}[\widetilde{C_{k}}, \widetilde{C_{s}}]$.
Since $c(G) \geq 3$, $\widetilde{G}$ contains one of the types $T_{j}$ for $j =5, 6, 7, 8$. As shown in Fig. 2, there exists a vertex $u$ of $T_{j}$ $(j =5, 6, 7, 8)$ such that $T_{j}-u$ also contains two vertex-joint cycles.
Hence there exists a vertex $u$ in the cycle of $\widetilde{G}$ such that
$\widetilde{G}-u$ does not does not satisfy Lemma \ref{mixed1}(i), thus
$$\eta(\widetilde{G}-u) \neq |V(G-u)|-2m(G-u)+2c(G-u).$$


{\bf  Subcase 2.2.} $\widetilde{G}$ does not satisfy Lemma \ref{mixed1}(ii) but satisfies Lemma \ref{mixed1}(i).

Then there exists at least one mixed cycle $\widetilde{C_{l}}$ is not even or $\widetilde{C_{l}}$ is even but $\sigma(\widetilde{C_{l}}) \not\equiv   l \, (\rm{mod} \ 4)$.
Since $c(G) \geq 2$ and $\widetilde{G}$ satisfies Lemma \ref{mixed1}(i), there exists a vertex $u$ in another cycle of $\widetilde{G}$ such that $\widetilde{G}-u$ does not satisfy Lemma \ref{mixed1}(ii).
By Lemma \ref{mixed1}, we have
$$\eta(\widetilde{G}-u) \neq |V(G-u)|-2m(G-u)+2c(G-u).$$

{\bf  Subcase 2.3.} $\widetilde{G}$ does not satisfy Lemma \ref{mixed1}(iii) but satisfies (1) and (ii) of Lemma \ref{mixed1}.

By the fact that $m(T_{G}) \geq m([T_{G}])$ and $m(T_{G}) \neq m([T_{G}])$, then $m(T_{G}) > m([T_{G}]) \geq 0$.
If $E(T_{G})=\emptyset$, then $\widetilde{G}$ is the union of some vertex-disjoint mixed cycles and isolated vertices and $m(T_{G})=0$. Which contradicts to $m(T_{G}) >  0$.
Therefore, one can suppose $E(T_{G}) \neq \emptyset$ in the following. In $T_{G}$, for every maximum matching $M$ of $T_{G}$, $M$ must contain at least one pendent edge of $T_{G}$.
Otherwise, one can find an $M$-augmenting path in $T_{G}$, which contradicts to Lemma \ref{MATCH}.
Let $x$ be a pendent vertex of $T_{G}$. Since $\widetilde{G}$ has no pendant vertex, $x \in W_{G}$. It can be checked that the cycle in $\widetilde{G}$ corresponding
to $x$ in $T_{G}$ is a pendant cycle, denoted by $\widetilde{C_{q}}$. Let $y$ be the unique vertex with degree three in $\widetilde{C_{q}}$ and $u$ be a vertex in the cycle $\widetilde{C_{q}}$.
Then $T_{G-u}$ is the graph obtained from $T_{G}$ and $\widetilde{C_{q}}-u$ by identifying $x$ and $y$ as one vertex. The following two
subcases can be identified for this case.

{\bf  Subcase 2.3.1.} Every maximum matching of $T_{G}$ cover all pendant edges of $T_{G}$.

One can suppose that $u$ be one of two vertices of $\widetilde{C_{q}}$ such that $u$ is adjacent to $y$. Note that $\widetilde{C_{q}}$ is an even cycle,
then the length of $\widetilde{C_{q}}-u-y$ is odd and $\widetilde{C_{q}}-u-y$ has a
perfect matching. By the definition of $T_{G}$, one has that the maximum matching of $T_{G-u}$ is the union of the maximum matching of $T_{G}$ and the maximum matching of $\widetilde{C_{q}}-u-y$. Then, one has
$$m(T_{G-u})=m(T_{G})+m(\widetilde{C_{q}}-u-y).$$

Hence every maximum matching of $T_{G-u}$ must covers some vertex in $W_{G-u}$. Then $m(T_{G-u}) > m([T_{G-u}])$.
By Lemma \ref{mixed1}(iii), we have
$$\eta(\widetilde{G}-u) \neq |V(G-u)|-2m(G-u)+2c(G-u).$$

{\bf  Subcase 2.3.2.} There exist some pendant edge, say $wt$, and some maximum matching, say $M(T_{G})$, of $T_{G}$ such
that $wt \notin M(T_{G})$ and $t$ is a pendant vertex of $T_{G}$.

Let $u'$ be a vertex of $\widetilde{C_{p}}$ such that $d(t_{0}, u') = 2$, where the cycle $\widetilde{C_{p}}$ of $\widetilde{G}$ corresponding to the pendant vertex $t$ of $T_{G}$, and $t_{0}$
is the unique vertex with degree three in $\widetilde{C_{p}}$. By the definition of $T_{G-u'}$, the maximum matching of $T_{G-u'}$ is the union of $M(T_{G})$ and the maximum matching of $\widetilde{C_{p}}-u'$.
Then, one has
$$m(T_{G-u'})=m(T_{G})+m(\widetilde{C_{p}}-u').$$

By Lemma \ref{MATCH}, $M(T_{G})$ must cover some vertices in $W_{G-u'}$. Then each maximum matching of $T_{G-u'}$ must cover some vertices in $W_{G-u'}$, i.e.,
$$m(T_{G-u'})\neq m([T_{G-u'}]).$$
By Lemma \ref{mixed1}(iii), one has
$$\eta(\widetilde{G}-u') \neq |V(G-u')|-2m(G-u')+2c(G-u').$$

This completes the proof.
\end{proof}

\begin{lemma}\label{L885}
Let $\widetilde{G}$ be a mixed graph without pendant vertices. Then
$$\eta(\widetilde{G}) \neq |V(G)|-2m(G)+2c(G)-1.$$
\end{lemma}
\begin{proof} If $G \cong D$ where $D$ is shown in Fig. 3., the result holds immediately from Lemma \ref{L880}. Then one can suppose that $G \not\cong D$ in the following.
We argue by induction on $c(G)$ to show the result.

If $c(G)=0$, then $\widetilde{G}$ is a forest. By Lemma \ref{L131}, $\eta(\widetilde{G}) = |V(G)|-2m(G)$, thus
$$\eta(\widetilde{G}) \neq |V(G)|-2m(G)+2c(G)-1.$$

If $c(G)=1$, then $\widetilde{G}$ is mixed unicyclic graph. By Lemma \ref{UNM}, one has that
$$\eta(\widetilde{G}) \neq |V(G)|-2m(G)+2c(G)-1.$$

Therefore one can assume that $c(G) \geq 2$ and the conclusion is true for $c(G) \leq k$. Next, we just need to prove the result is true for $c(G) = k+1$.
Suppose on the contrary, there exists some mixed graph $\widetilde{H}$ without pendant vertices such that $c(H) = k+1$ and $\eta(\widetilde{H}) = |V(H)|-2m(H)+2c(H)-1.$

Let $x$ be any vertex on some cycle of $\widetilde{H}$. For the mixed graph $\widetilde{H}-x$, by Lemmas \ref{L17} and \ref{L23}, we have
$$m(H) \leq m(H-x)+1,$$
$$c(H) \geq c(H-x)+1,$$
and
$$|V(H)|=|V(H-x)|+1.$$
By Lemma \ref{L0001}, one has
\begin{eqnarray*}
\eta(\widetilde{H}-x)+1
&\geq&\eta(\widetilde{H})\\
&=&|V(H)|-2m(H)+2c(H)-1\\
&\geq&|V(H-x)|+1-2[m(H-x)+1]+2[c(H-x)+1]-1\\
&=&|V(H-x)|-2m(H-x)+2c(H-x).
\end{eqnarray*}
By Lemma \ref{mixed1}, for any vertex $x$ in some cycle $\widetilde{H}$, one has
\begin{equation} \label{E1}
\eta(\widetilde{H}-x)=|V(H-x)|-2m(G-x)+2c(G-x)-s, s=0,1.
\end{equation}

We will deal with the problem with two subcases according to the pendant vertices of $\widetilde{G}$.

{\bf  Case 1.} $\widetilde{H}-x$ contains no pendant vertex.

Since $c(H-x)\leq c(H)-1=k$, by the induction hypothesis,
$$\eta(\widetilde{H}-x) \neq |V(H-x)|-2m(H-x)+2c(H-x) -1.$$
Then, from (\ref{E1}), one has that
\begin{equation} \label{E2}
\eta(\widetilde{H}-x) = |V(H-x)|-2m(H-x)+2c(H-x)
\end{equation}
holds for any vertex $x$ in the cycle of $\widetilde{H}$.

On the other hand, since $\eta(\widetilde{H}) = |V(H)|-2m(H)+2c(H)-1$, i.e., $\eta(\widetilde{H}) \neq |V(H)|-2m(H)+2c(H)$, by Lemma \ref{L884},
there exists a vertex $v$ in the cycle of $\widetilde{H}$, such that
$$\eta(\widetilde{H}-v) \neq |V(H-v)|-2m(H-v)+2c(H-v),$$
which contradicts to (\ref{E2}).

{\bf  Case 2.} $\widetilde{H}-x$ contains some pendant vertices.

{\bf  Subcase 2.1.} $\widetilde{H}-x$ contains at least one pendant vertex whose quasi-pendant vertex lies in some cycle of $\widetilde{H}-x$.

By Lemma \ref{L883}, for any vertex $x$ in the cycle of $\widetilde{H}$, we have
$$\eta(\widetilde{H}-x) \leq  |V(H-x)|-2m(H-x)+2c(H-x)-2,$$
which contradicts to (\ref{E1}).

{\bf  Subcase 2.2.} All the quasi-pendant vertices of $\widetilde{H}-x$ lie in no cycle of $\widetilde{H}-x$.

A subgraph $\widetilde{H_{1}}$ of $\widetilde{H}-x$ can be obtained after deleting all the pendant vertices and their quasi-pendant vertices of $\widetilde{H}-x$.
If $\widetilde{H_{1}}$ contains no pendant vertices or at least a pendant vertex whose quasi-pendant vertex lies in some cycle, then $\widetilde{H_{1}}$
is as we required and we are done. Otherwise, a subgraph $\widetilde{H_{2}}$ of $\widetilde{H_{1}}$ can be obtained after deleting all the pendant vertices and their quasi-pendant vertices of $\widetilde{H_{1}}$.
If $\widetilde{H_{2}}$ contains no pendant vertices or at least a pendant vertex whose quasi-pendant vertex lies in some cycle,
then $\widetilde{H_{2}}$ is as we required and we are done. Otherwise, repeating the above steps until we obtain a mixed graph $\widetilde{H_{0}}$ that meets the requirements.

{\bf  Subcase 2.2.1.} $\widetilde{H_{0}}$ contains no pendant vertices.

Since $c(H_{0})=c(H-x)\leq c(H)-1=k$, by the induction hypothesis,
$$\eta(\widetilde{H_{0}}) \neq |V(H_{0})|-2m(H_{0})+2c(H_{0})-1.$$
Since $\widetilde{H_{0}}$ is obtained from $\widetilde{H}-x$ by removing a series pendant vertices whose quasi-pendant vertices do not lie in any cycle and their quasi-pendant vertices, by Lemma \ref{L771}, one has
$$\eta(\widetilde{H}-x) \neq |V(H-x)|-2m(H-x)+2c(H-x)-1.$$

From (\ref{E1}), one has
\begin{equation} \label{E3}
\eta(\widetilde{H}-x) = |V(H-x)|-2m(H-x)+2c(H-x)
\end{equation}
holds for any vertex $x$ in the cycle of $\widetilde{H}$.

On the other hand, since $\eta(\widetilde{H}) = |V(H)|-2m(H)+2c(H)-1$, i.e., $\eta(\widetilde{H}) \neq |V(H)|-2m(H)+2c(H)$, by Lemma \ref{L884},
there exists a vertex $u$ in the cycle of $\widetilde{H}$, such that
$$\eta(\widetilde{H}-u) \neq |V(H-u)|-2m(H-u)+2c(H-u),$$
which contradicts to (\ref{E3}).

{\bf  Subcase 2.2.2.} $\widetilde{H_{0}}$ contains at least one pendant vertex whose quasi-pendant vertex lies in some cycle.

By Lemma \ref{L883}, we have
$$\eta(\widetilde{H_{0}}) \leq  |V(H_{0})|-2m(H_{0})+2c(H_{0})-2.$$
Then,
$$\eta(\widetilde{H_{0}}) \neq |V(H_{0})|-2m(H_{0})+2c(H_{0})-1.$$
Since $\widetilde{H_{0}}$ is obtained from $\widetilde{H}-x$ by removing a series pendant vertices (whose quasi-pendant vertices do not lie in any cycle) and their quasi-pendant vertices, by Lemma \ref{L771}, one has
$$\eta(\widetilde{H}-x) \neq |V(G-x)|-2m(G-x)+2c(G-x)-1.$$
Then, from (\ref{E1}), one has
\begin{equation} \label{E4}
\eta(\widetilde{H}-x) = |V(H-x)|-2m(H-x)+2c(H-x)
\end{equation}
holds for any vertex $x$ in the cycle of $\widetilde{H}$.

On the other hand, since $\eta(\widetilde{H}) = |V(H)|-2m(H)+2c(H)-1$, i.e., $\eta(\widetilde{H}) \neq |V(H)|-2m(H)+2c(H)$, by Lemma \ref{L884},
there exists a vertex $u$ in the cycle of $\widetilde{H}$, such that
$$\eta(\widetilde{H}-u) \neq |V(H-u)|-2m(H-u)+2c(H-u),$$
which contradicts to (\ref{E4}).

This completes the proof.
\end{proof}


Now, we give the proof of the main result of this section.

{\bf The proof of Theorem \ref{T1}.}
If $\widetilde{G}$ is an acyclic mixed graph, by Lemma \ref{L131}, the result follows. In the following, we suppose that $\widetilde{G}$ contains at least one cycle.

{\bf  Case 1.} $\widetilde{G}$ has no pendant vertices.

The result can be obtained from Lemma \ref{L885} immediately.

{\bf  Case 2.} $\widetilde{G}$ has some pendant vertices.

We will deal with the problem with two subcases according to the quasi-pendant vertices of $\widetilde{G}$.

{\bf  Subcase 2.1.} All quasi-pendant vertices of $\widetilde{G}$ do not lie in any cycle.

A subgraph $\widetilde{G_{1}}$ of $\widetilde{G}$ can be obtained after deleting all the pendant vertices and their quasi-pendant vertices of $\widetilde{G}$.
If $\widetilde{G_{1}}$ contains no pendant vertices or at least a pendant vertex whose quasi-pendant vertex lies in some cycle, then $\widetilde{G_{1}}$
is as we required and we are done. Otherwise, a subgraph $\widetilde{G_{2}}$ of $\widetilde{G_{1}}$ can be obtained after deleting all the pendant vertices and their quasi-pendant vertices of $\widetilde{G_{1}}$.
If $\widetilde{G_{2}}$ contains no pendant vertices or at least a pendant vertex whose quasi-pendant vertex lies in some cycle,
then $\widetilde{G_{2}}$ is as we required and we are done. Otherwise, repeating the above steps until we obtain a mixed graph $\widetilde{G_{0}}$ that meets the requirements.

{\bf  Subcase 2.1.1.} $\widetilde{G_{0}}$ contains no pendant vertices.

By Lemma \ref{L885}, we have
$$\eta(\widetilde{G_{0}}) \neq |V(G_{0})|-2m(G_{0})+2c(G_{0})-1.$$
Since $\widetilde{G_{0}}$ is obtained from $\widetilde{G}$ by removing a series pendant vertices whose quasi-pendant vertices do not lie in any cycle and their quasi-pendant vertices, by Lemma \ref{L771}, one has
$$\eta(\widetilde{G}) \neq |V(G)|-2m(G)+2c(G)-1.$$

{\bf  Subcase 2.1.2.} $\widetilde{G_{0}}$ contains at least one pendant vertex whose quasi-pendant vertex lies in some cycle.

By Lemma \ref{L883}, we have
$$\eta(\widetilde{G_{0}}) \leq  |V(G_{0})|-2m(G_{0})+2c(G_{0})-2.$$
Then,
$$\eta(\widetilde{G_{0}}) \neq |V(G_{0})|-2m(G_{0})+2c(G_{0})-1.$$
Since $\widetilde{G_{0}}$ is obtained from $\widetilde{G}$ by removing a series pendant vertices (whose quasi-pendant vertices do not lie in any cycle) and their quasi-pendant vertices, by Lemma \ref{L771}, one has
$$\eta(\widetilde{G}) \neq |V(G)|-2m(G)+2c(G)-1.$$

{\bf  Subcase 2.2.} There exists at least one pendant vertex of $\widetilde{G}$ whose quasi-pendant vertex lies in some cycle.

By  Lemma \ref{L883}, we have
$$\eta(\widetilde{G}) \neq |V(G)|-2m(G)+2c(G)-1.$$

This completes the proof of Theorem \ref{T1}.\hspace{8cm}$\square$

\section{Proof of Theorem \ref{T2}.}

In this section, the proof of Theorem \ref{T2} is presented. For a fixed integer $c(G)$, some mixed graphs are constructed to show that there are infinitely many connected mixed graphs with the nullity $\eta(\widetilde{G})=|V(G)|-2m(G)+2c(G)-k$, where $0 \leq k \leq 3c(G)$ and $k \neq 1$.

{\bf The proof of Theorem \ref{T2}.}
Let $s_{1}, s_{2}$ and $s_{3}$ be any integers with $s_{1}+s_{2}+s_{3}=c(G)$.
Let $\widetilde{K}_{1, c(G)+1}$ be a mixed star, $u$ be the center vertex of $\widetilde{K}_{1, c(G)+1}$ and $v_{1}, v_{2}, \cdots, v_{c(G)+1}$ be the pendant
vertices of $\widetilde{K}_{1, c(G)+1}$, respectively. Let $\widetilde{C^{i}_{3}} $ be a mixed cycle with size 3 and $\sigma(\widetilde{C^{i}_{3}})$ is even for $1 \leq i \leq s_{1} $,
$\widetilde{C^{j}_{4}} $ be a mixed cycle with size 4 and $\sigma(\widetilde{C^{j}_{4}}) \equiv 0 \ (\rm{mod} \ 4) $ for $1 \leq j \leq s_{2} $, and $\widetilde{H_{l}} $ be any mixed graph with the underlying graph
$H_{l}$ obtained from a cycle of size 4 by attaching one pendant edge for $1 \leq l \leq s_{3} $.
Let  $\widetilde{G}$ be the mixed graph which constructed by identifying $v_{i}$ with a vertex of $\widetilde{C^{i}_{3}} $, $v_{j+s_{1}}$ with a vertex of $\widetilde{C^{j}_{4}} $,
$v_{l+s_{1}+s_{2}}$ with a pendant vertex of $\widetilde{H_{l}} $, where $i=1, 2, \cdots, s_{1}$, $j=1, 2, \cdots, s_{2}$ and $l=1, 2, \cdots, s_{3}$.
Then
$$V(G)=3s_{1}+4s_{2}+5s_{3}+2,$$
$$m(G)=s_{1}+2s_{2}+2s_{3}+1,$$
$$c(G)=s_{1}+s_{2}+s_{3}.$$

By Lemma \ref{L25}, one has $\eta(\widetilde{C^{i}_{3}})=0$ and $\eta(\widetilde{C^{j}_{4}})=2$ for $1 \leq i \leq s_{1} $ and $1 \leq j \leq s_{2} $.
It can be checked that the result graph of $\widetilde{H_{l}} $ by deleting the unique pendant vertex and its quasi-pendant vertex is a mixed star $\widetilde{K}_{1, 2}$.
By Lemma \ref{L13} and $\eta(\widetilde{K}_{1, 2})=1$, we have $\eta(\widetilde{H_{l}} )=1$ for $1 \leq l \leq s_{3} $.
Then, by Lemma \ref{L13}, one has
\begin{eqnarray*}
\eta(\widetilde{G})
&=&\eta(\widetilde{G}-u-v_{c(G)+1})\\
&=&2s_{2}+s_{3}\\
&=&|V(G)|-2m(G)+2c(G)-(3s_{1}+2s_{3}).
\end{eqnarray*}
Since $s_{i}\geq 0$,  $3s_{1}+2s_{3}$ can take over every integer from zero to $3c(G)$ except for one.

This completes the proof of Theorem \ref{T2}.\hspace{8cm}$\square$


\section*{Acknowledgments}

This research is supported by National Natural Science Foundation of China (Nos.11971054, 11731002, 11771039, 11771443).

\end{document}